\documentclass[11 pt,a4paper]{amsart} 
\usepackage{amsfonts,amssymb,amscd,amsmath,enumerate,verbatim,calc} 
\usepackage{amsthm}
\usepackage[all,cmtip]{xy}

\newtheorem{theorem}{Theorem}[section]
\newtheorem{definition}{Definition}[section]

\newtheorem{lemma}[theorem]{Lemma}

\newtheorem{proposition}[theorem]{Proposition}
\newtheorem{remark}[theorem]{Remark}
\usepackage{tikz-cd}
\usepackage{mathtools}

\newtheorem{example}{Example}[section]
\numberwithin{equation}{section}

\usepackage{amsfonts}

\renewcommand{\c}{\circ}

\renewcommand{\ker}{\textnormal{ker}}

\newcommand{\gyr}{\textnormal{gyr}}

\begin{document}
	\title{Gyrogroup extensions and semi cross product}

\author{Akshay Kumar$^{1}$,
Mani Shankar Pandey$^{2}$,  Seema Kushwaha$^{3}$ AND \\Sumit Kumar Upadhyay$^{4}$\vspace{.4cm}\\
{Department of Applied Sciences,\\ Indian Institute of Information Technology Allahabad\\Prayagraj, U. P., India} }
%{$^\dagger$Department of Mathematics,\\ University of Allahabad\\Prayagraj, U. P., India}}

%\thanks{$^\#$Retired Professor}

\thanks{$^1$mehtaaksh4@gmail.com, $^2$manishankarpandey4@gmail.com, $^3$seema28k@gmail.com, $^4$upadhyaysumit365@gmail.com}

		
	\begin{abstract} 
	The aim of the article is to study the extension theory of gyrogroups under certain conditions. Consequently, we see that there is an equivalence between the category GEXT of group-gyro extensions and the category GFAC of group-gyro factor systems.  We also give a notion of semi cross product of a group and  a gyrogroup which is a construction method of a larger class of gyrogroups. 
	\end{abstract}
	
\maketitle
\section{Introduction}
In 1988,  Ungar introduced a non-associative algebraic structure, {\it gyrogroup,} which is a generalization of groups. 
The concept of gyrogroups first came into the picture
in the study of Einstein’s relativistic velocity addition law \cite{ungar_thomas,ungar_book}. In analogy to the group theory, Foguel and Ungar (\cite{foguel_involutory}, Definition 4.8 and Lemma 4.9) gave a concept of a normal subgroup of a gyrogroup such that quotient of a gyrogroup by a normal subgroup is a gyrogroup. Thus for a normal subgroup $H$ of a gyrogroup $G$, we have the following short exact sequence \[\xymatrix{
E \equiv \{e\}\ar[r] & H\ar[r]^{i} & G \ar[r]^{\nu} & G/H\ar[r] & \{e\}
},
\] where $i$ is the inclusion map and $\nu$ is the natural quotient gyrogroup homomorphism. 

Let $G$ and $K$ be two gyrogroups. Let $\phi: G \to K$ be a gyrogroups homomorphism. In 2015, authors proved that $G/\ker(\phi)$  is a gyrogroup (\cite{suksumran_isomorphism}, Theorem 27). Thus we have the following short exact sequence \[\xymatrix{
E \equiv \{e\}\ar[r] & \ker(\phi)\ar[r]^{i} & G \ar[r]^{\phi} & K\ar[r] & \{e\}
}.
\]
By the above two discussions, one can talk about short exact sequences of gyrogroups.
Schreier's extension theory for groups was developed to classify all groups $G$ having $H$ as a normal subgroup such that ${G}/{H}$ is isomorphic to $K$, for any two groups $H$ and $K$ (for details, see \cite{rjl_algebra2}). In analogous to groups, we pose the following questions:
\begin{enumerate}
	\item For a given group $H$,  classify all gyrogroups $G$ such that $G/H$ is a gyrogroup.
	\item For a given gyrogroup $H$,  classify all gyrogroups $G$ such that $G/H$ is a group.
	\item For a given gyrogroup $H$,  classify all gyrogroups $G$ such that $G/H$ is a gyrogroup.
\end{enumerate}

In \cite{Bruck}, R. H. Bruck discussed extensions of loops. With the help of these extensions, more examples of gyrocommutative gyrogroups can be obtained.   In 2000, Rozga \cite{rozga} discussed the central extensions for gyrocommutative gyrogroups. The author found that at least for a class of gyrocommutative gyrogroups, a cocycle equation emerges in a natural way. 
Recently, Lal and Kakkar \cite{vkakkar} discussed a few results on extensions for group based gyrogroups. In Section 3 of the article,  we try to develop Schreier's extension theory to classify all gyrogroups $G$ having $H$ as a normal subgroup such that ${G}/{H}$ is isomorphic to $K$, for any group $H$ and gyrogroup $K$. That is, we discuss the first question. 

Section 4 is dedicated to constructions of gyrogroups. For this, we introduce a notion of semi cross product $H\Join K$, for a group $H$ and a gyrogroup $K$. We also show that for any split extension \[
\xymatrix{
E \equiv \{e\}\ar[r] & H\ar[r]^{i} & G \ar[r]^{\beta} & K\ar[r] & \{e\}
}
\] of $H$ by $K$, $G\cong H\Join K$. More precisely, we give examples of gyrogroups of order 24, 32 as well as of infinite order.

\section{preliminaries}
In this section, we give basic definition and terminology of gyrogroups that will be used later in this article. 

\begin{definition} \label{gyrogroup}  A pair $(G,\c)$ consisting of a non-empty set $G$ and a binary operation $\c$ on $G$ is called a {\it gyrogroup} if its binary operation $\c$ satisfies the following axioms. 
	\begin{enumerate}
		\item There exists an element $e\in G$, called
		a left identity, such that $$e\c a=a\,\,\text{ for\, all } a\in G.$$
\item For each $a\in G,$ there exists an element $a^{-1}\in G$, called a left inverse of $a,$ such that $a^{-1}\c a=e$.
\item For $a,b\in G,$ there exists an automorphism $\gyr[a,b]$ such that 
$$ a\c(b\c c)= (a\c b)\c \gyr[a,b](c)$$
for all $a,b,c\in G.$
\item $\gyr[a\c b, b]=\gyr[a,b]$ for all $a,b\in G.$
	\end{enumerate} 
	\end{definition}
	\begin{definition}(\cite{foguel_involutory})
	 A non-empty subset $H$ of a gyrogroup $(G,\c)$ is a subgroup if it is a group under the restriction of $\c$ to $H$.
	\end{definition}
	\begin{definition}(\cite{foguel_involutory}) A subgroup $H$ of a gyrogroup $(G,\c)$ is normal in $G$ if
	\begin{enumerate}
	\item $\gyr[g, h] = I_G$ for all $h\in H$ and $g\in G$;
	\item $\gyr[g, g'](H) \subseteq (H)$ for all $g, g'\in G$;
	\item $g \c H = H \c g$ for all $g\in G$.
	\end{enumerate}
	\end{definition}
	\begin{lemma}(\cite{foguel_involutory})
	If $H$ is a normal subgroup of a gyrogroup $G$, then $G/H$ forms a factor gyrogroup.
	\end{lemma}
Now, we state some gyrogroup identities which will be used further.
\begin{proposition} \cite{ungar_book} \label{gyro identity}
Let $(G, \c)$ be a gyrogroup. Then, for all $a,b, c\in G,$
\begin{enumerate}
\item $(a^{-1})(ab)=b$ (left cancellation);
\item $ab=e\Leftrightarrow ba=e$;
\item $\gyr[ab,a^{-1}]=\gyr[a,b]$;
\item $\gyr[a,b](c)=(ab)^{-1}(a(bc))$ (gyrator identity);
%\item 	$\gyr[a b,\gyr[a,b](c)]\gyr[a,b]=\gyr[a,b c]\gyr[b,c]$
\item  $\gyr^{-1}[a, b]=\gyr[b,a]$;
\item $(ab)^{-1} = \gyr[a,b](b^{-1}a^{-1})$.
\end{enumerate}
\end{proposition}
\begin{definition} 
\begin{enumerate}
\item A short exact sequence 
\[
\xymatrix{
E \equiv \{e\}\ar[r] & H\ar[r]^{i} & G \ar[r]^{\beta} & K\ar[r] & \{e\}
}
\]
of gyrogroups is called an extension of $H$ by $K$, where $i$ is the inclusion map. A map $t: K \longrightarrow G$ is called a section of $E$ if $\beta t=Id_K$ and $t(e)=e$.
\item A morphism from an extension of gyrogroups
\[
\xymatrix{
E \equiv \{e\}\ar[r] & H\ar[r]^{i_1} & G \ar[r]^{\beta} & K\ar[r] & \{e\}
}
\]
to another extension of gyrogroups
\[
\xymatrix{
E' \equiv \{e\}\ar[r] & H'\ar[r]^{i_2} & G' \ar[r]^{\beta'} & K'\ar[r] & \{e\}
}
\]
is a triple $(\lambda, \mu, \nu)$, where $\lambda: H \longrightarrow H'$, $\mu: G \longrightarrow G'$ and $\nu : K \longrightarrow K'$ are gyrogroup homomorphisms such that the following diagram 
\begin{center}
\[
\xymatrix{
E \equiv \{e\}\ar[r] & H \ar[r]^{i_1} \ar[d]^{\lambda} & G \ar[r]^{\beta} \ar[d]^{\mu} & K \ar[r] \ar[d]^{\nu} & \{e\}\\
  E' \equiv \{e\} \ar[r] & H' \ar[r]^{i_2} & G' \ar[r]^{\beta'} & K' \ar[r] & \{e\}
}
\]
\end{center} 
is commutative. 
%\item Two extensions $E_1$ and $E_2$ of $H$ by $K$ are equivalent if there is an isomorphism $\phi :G_1 \longrightarrow G_2$ such that the following
%diagram
%\[
%\xymatrix{
%E_1 \equiv \{e\}\ar[r] & H \ar[r]^{i_1} \ar[d]^{Id_H} & G_1 \ar[r]^{\beta_1} \ar[d]^{\phi} & K \ar[r] \ar[d]^{Id_K} & \{e\} \\
%  E_2 \equiv \{e\} \ar[r] & H \ar[r]^{i_2} & G_2 \ar[r]^{\beta_2} & K \ar[r] & \{e\}
%}
%\]
%is commutative.
\end{enumerate}
\end{definition}

%\begin{proposition}\cite[Proposition 6]{suksumran_isomorphism} \label{l-subgyro} Let $G$ be a gyrogroup and let $X \subseteq G$. Then the following are equivalent\\
%1) $\gyr[a, b](X) \subseteq X$ for all $a, b \in G$,\\
%2) $\gyr[a, b](X) = X$ for all $a, b \in G$.
%\end{proposition}

\section{Extension Theory of a group by a gyrogroup}
	Let	
	$\begin{tikzcd}
		E\equiv \{e\}\arrow{r}  & H \arrow{r}{i} &G\arrow{r}{\beta} &K\arrow{r}&\{e\}
	\end{tikzcd}$
be a gyrogroup extension of $H$ by $K$. Let $t$ be a section of $E$ and $g\in G$. Then 
$$\beta(t(\beta(g))g^{-1})=e \Rightarrow t(\beta(g))g^{-1}\in H.$$
Thus, there exists $h\in H$ such that 
\begin{align*}
&t(\beta(g))g^{-1}=h^{-1}\\
&\Rightarrow h(t(\beta(g))g^{-1})=e\\
&\Rightarrow (ht(x))\gyr [h,t(x)](g^{-1})=e, \ \text{where}\ \beta(g)=x\in K \\
&\Rightarrow (\gyr [h,t(x)](g^{-1}))((ht(x)))=e ~(\text{since}~ ab = e \Leftrightarrow ba = e)\\
&\Rightarrow ht(x)=\gyr [h,t(x)](g)\\
&\Rightarrow g=\gyr [t(x),h](ht(x)).
\end{align*}
\noindent
Therefore, given a  gyrogroup extension 	
$$\begin{tikzcd}
E\equiv \{e\}\arrow{r}  & H \arrow{r}{i} &G\arrow{r}{\beta} &K\arrow{r}&\{e\}	\end{tikzcd}$$ of $H$ by $K$ with a choice of section $t$, every $g\in G$ can be  written as
$$ g=\gyr [t(x),h](ht(x)), \text{for some $h\in H$ and $x\in K$}.$$ 
It is easy to see that this representation is unique. Now, we study gyrogroup extensions under a special condition.
\begin{definition}
A gyrogroup extension $\begin{tikzcd}
		E\equiv \{e\}\arrow{r}  & H \arrow{r}{i} &G\arrow{r}{\beta} &K\arrow{r}&\{e\}
	\end{tikzcd}$ of $H$ by $K$ is said to be a group-gyro extension if $\gyr[h,g]=I_G$, for all $h\in H$ and $g\in G.$ In fact, $H$ is a normal subgroup of the gyrogroup $G$.
\end{definition}
It can be easily seen that	if $\begin{tikzcd}
		E\equiv \{e\}\arrow{r}  & H \arrow{r}{i} &G\arrow{r}{\beta} &K\arrow{r}&\{e\}
	\end{tikzcd}$ is a group-gyro extension of $H$ by $K$ and $t$ is a section of $E$, then any $g\in G$ can be uniquely expressed as $g=ht(x)$.

Also, for each $h\in H$ and $x\in K,$  we have $(t(x)h)t(x)^{-1}=t(x)(ht(x)^{-1})\in H$. Thus, for $x\in K,$ there exists a map $\sigma_x^{t}:H\to H$ given by
\begin{equation}\label{1}
\sigma_x^t(h)=(t(x)h)t(x)^{-1}.
\end{equation}
In fact,  $\sigma_x^t\in Aut(H)$   for each $x\in K.$
Let $x,y\in K.$ Then 
\begin{align*}
&\beta(t(xy)(t(x)t(y))^{-1})=e\\
&\Rightarrow t(xy)(t(x)t(y))^{-1} = f^t(x, y)^{-1} ~\text{for some}~ f^t(x, y) \in H\\
&\Rightarrow  f^t(x, y)(t(xy)(t(x)t(y))^{-1}) = e\\
&\Rightarrow (f^t(x, y)t(xy))(t(x)t(y))^{-1}) =e\\
&\Rightarrow (t(x)t(y))^{-1})(f^t(x, y)t(xy)) =e\\
&\Rightarrow  t(x)t(y) = f^t(x, y)t(xy).
\end{align*}
Thus, there exists a map $f^t: K\times K\to H$ such that
\begin{equation}\label{2}
t(x)t(y)=f^t(x,y)t(xy).
\end{equation}
Since $t(e)=(e)$,
\begin{equation}\label{3}
f^t(x,e)=e=f^t(e,y) \ \text{for each}\  x,y\in K.
\end{equation}
Also, for $h\in H$ and $x,y\in K$
\begin{align*}
	t(x)(h t(y))&=(t(x) h) \gyr[t(x),h](t(y))\\
	&=(t(x) h t(x)^{-1} t(x)) t(y)\\
	&=((t(x) h t(x)^{-1} )t(x)) t(y) ~~\,(\text{since}~   \gyr[t(x)h, t(x)^{-1}] = I_G)\\
	&=(\sigma_x^t(h)t(x))t(y)\\
	&=\sigma_x^t(h)(t(x) \gyr[t(x), \sigma_x^t(h)](t(y)))\\
	&=\sigma_x^t(h)(t(x) t(y))\\
	&=\sigma_x^t(h)(f^t(x,y) t(xy))=(\sigma_x^t(h) f^t(x,y)) t(xy).
\end{align*}
Similarly,
\begin{align*}
(t(x)h) t(y)	&=(\sigma_x^t(h) f^t(x,y)) t(xy)=	t(x)(h t(y)).
\end{align*}
	Further,
	\begin{align*}
		(h t(x)) t(y)&=h(t(x) \gyr[t(x),h](t(y)))\\
		&=h(t(x) t(y))=h(f^t(x,y) t(xy))=(hf^t(x,y)) t(xy).
	\end{align*}
	Now, it is easy to observe that
%		\begin{align*}
%		(h t(x)) (k t(y))&=h(t(x)(kt(y)))=h((\sigma_x^t(k)f^t(x,y)) t(xy))=(h\sigma_x^t(k)f^t(x,y)) t(xy)
%	\end{align*}
	\begin{equation}\label{4}
	(h t(x)) (k t(y))=(h\sigma_x^t(k)f^t(x,y)) t(xy).	
	\end{equation}
By Equation \ref{2}, for $x\in K$ we have
\begin{align*}
&t(x^{-1})t(x)=f^t(x^{-1},x)\\
&\Rightarrow (f^t(x^{-1},x))^{-1}(t(x^{-1})t(x))=e\\
&\Rightarrow t(x)(f^t(x^{-1},x)^{-1}t(x^{-1}))=e\\
&\Rightarrow t(x)^{-1}=f^t(x^{-1},x)^{-1}t(x^{-1}).
\end{align*}
So, we have the following equation
\begin{equation}\label{5}
 t(x)^{-1}=f^t(x^{-1},x)^{-1}t(x^{-1}).
\end{equation}

Now, let $h,k,l\in H$ and $x,y,z\in K$. Then by the gyrator identity, we have
\begin{align*}
&\gyr[ht(x),kt(y)](lt(z))=((ht(x))(kt(y))^{-1}((ht(x))((kt(y)(lt(z)))\\
&= (h\sigma_x^t(k)f^t(x,y)t(xy))^{-1}((ht(x))(k\sigma^t_y(l)f^t(y,z)t(yz)))\\
&=(t(xy)^{-1}f^t(x,y)^{-1}\sigma_y^t(k)^{-1}h^{-1})(h\sigma_x^t(k\sigma_y^t(l)f^t(y,z))f^t(x,yz)t(x(yz)))\\
&=f^t((xy)^{-1},xy)^{-1}t((xy)^{-1})f^t(x,y)^{-1}\sigma_x^t(\sigma_y^t(l)f^t(y,z))f^t(x,yz)t(x(yz))\\
&=f^t((xy)^{-1},xy)^{-1} (\sigma_{(xy)^{-1}}^t(f^t(x,y)^{-1}\sigma_x^t(\sigma_y^t(l)f^t(y,z))f^t(x,yz))t((xy)^{-1}))t(x(yz))\\
&=f^t((xy)^{-1},xy)^{-1} (\sigma_{(xy)^{-1}}^t(f^t(x,y)^{-1}\sigma_x^t(\sigma_y^t(l)f^t(y,z))f^t(x,yz)))f^t((xy)^{-1},(x(yz))t(\gyr[x,y](z))
\end{align*}
Here, for $x, y \in K$, we define a function $F^t_{(x,y)}:H\times K\longrightarrow H$ such that 

\begin{equation}\label{6}
F^t_{(x,y)}(l,z)=f^t((xy)^{-1},xy)^{-1} (\sigma_{(xy)^{-1}}^t(f^t(x,y)^{-1}\sigma_x^t(\sigma_y^t(l)f^t(y,z))f^t(x,yz)))f^t((xy)^{-1},x(yz)).
\end{equation}
Therefore, 
\begin{equation}\label{gyroF}
\gyr[ht(x),kt(y)](lt(z))=F^t_{(x,y)}(l,z)t(\gyr[x,y](z)).
\end{equation}

Now we see the properties of the function $F^t$ which is defined in Equation \ref{6}.

\begin{proposition}\label{propF} For $x, y \in K$, $F^t_{(x,y)}$ satisfies the following properties:
	\begin{enumerate}
		\item $F^t_{(x,e)}(l,z)=l=F^t_{(e,y)}(l,z), \ \forall \ l\in H\ \text{and}\ z\in K. $
		\item For $l_1,l_2\in H$ and $z_1,z_2\in K,$ $$F^t_{(x,y)}(l_1\sigma_{z_1}^t(l_2)f^t(z_1,z_2),z_1z_2)=F^t_{(x,y)}(l_1,z_1)\sigma_{(\gyr[x,y](z_1))}^t(F^t_{(x,y)}(l_2,z_2))f^t(\gyr[x,y](z_1),\gyr[x,y](z_2)).$$
		\item  $F^t_{(xy,y)}=F^t_{(x,y)}.$
%		\item If $s$ and $t$ are two sections of $E$. Then there exists a function  $g:K\to H,$  such that $s(x)=g(x)t(x)$ and  $$F^t_{(x,y)}(lg(z),z)= F^s_{(x,y)} (l,z )g(\gyr[x,y](z))$$ for $x,y,z\in K.$
	\end{enumerate}
\end{proposition}
\begin{proof}
\begin{enumerate}
\item Since $\gyr[h,kt(y)]=I_G=\gyr[ht(x),k]$, we have
	\begin{equation*}\label{F^t}
		F^t_{(x,e)}(l,z)=l=F^t_{(e,y)}(l,z), \ \forall \ l\in H\ \text{and}\ x,y,z\in H. 
	\end{equation*}
\item Note  that $\gyr[x,y]$ is an automorphism, hence
	\begin{align*}\label{auto}
	&	\gyr[ht(x),kt(y)]((l_1t(z_1))(l_2t(z_2)))\\
	&	=\gyr[ht(x),kt(y)](l_1t(z_1))\gyr[ht(x),kt(y)](l_2t(z_2))\\
	&	=(F^t_{(x,y)}(l_1,z_1)t(\gyr[x,y]z_1))(F^t_{(x,y)}(l_2,z_2)t(\gyr[x,y](z_2)))\\
		&=F^t_{(x,y)}(l_1,z_1)\sigma_{\gyr[x,y](z_1))}^t(F^t_{(x,y)}(l_2,z_2))f^t(\gyr[x,y](z_1),\gyr[x,y](z_2))t(\gyr[x,y](z_1z_2)).
	\end{align*}
	Also,
	\begin{align*}%\label{auto 2}
	\gyr[ht(x),kt(y)]((l_1t(z_1))(l_2t(z_2)))
	&	=\gyr[ht(x),kt(y)](l_1\sigma_{z_1}^t(l_2)f^t(z_1,z_2)t(z_1z_2))\\
	&=F^t_{(x,y)}(l_1\sigma_{z_1}^t(l_2)f^t(z_1,z_2),z_1z_2)t(\gyr[x,y](z_1z_2)).
 	\end{align*}
	On comparing both the expressions for the gyromap, we get
	\begin{equation*}\label{automorphism}
		F^t_{(x,y)}(l_1\sigma_{z_1}^t(l_2)f^t(z_1,z_2),z_1z_2)=F^t_{(x,y)}(l_1,z_1)\sigma_{(\gyr[x,y](z_1))}^t(F^t_{(x,y)}(l_2,z_2))f^t(\gyr[x,y](z_1),\gyr[x,y](z_2)).
	\end{equation*}
	\item 		Also,
	\begin{align*}\label{loop}
		\gyr[ht(x)kt(y),kt(y)](lt(z))&=\gyr[h\sigma_x^t(k)f^t(x,y)t(xy),kt(y)](lt(z))\\
	&= F^t_{(xy,y)}(l, z)t(\gyr[xy,y](z))
	\end{align*}
	\begin{align*}
	\gyr[ht(x),kt(y)](lt(z))
		&=F^t_{(x,y)}(l,z)t(\gyr[x,y](z))
			\end{align*}
Since
	$\gyr[ht(x)kt(y),kt(y)]=\gyr[ht(x),kt(y)]$ and 
	$\gyr[xy,y]= \gyr[x,y]$, $ F^t_{(xy,y)}=F^t_{(x,y)}$.
\end{enumerate}
	
\end{proof}

Also, let $h,k,l\in H$ and $x, y, z \in K$, 

$ht(x)((kt(y))(lt(z))) = ht(x)(k\sigma_y^t(l)f^t(y, z)t(yz))= h\sigma_x^t(k\sigma_y^t(l)f^t(y, z))(t(x)t(yz))= \\
h\sigma_x^t(k\sigma_y^t(l)f^t(y, z))f^t(x, yz)t(x(yz))$.

On the other hand, 

$(ht(x))(kt(y)))\gyr[ht(x),kt(y)](lt(z))= (h\sigma_x^t(k)f^t(x, y)t(xy))(F^t_{(x,y)} (l,z) t(\gyr[x,y](z)))=\\ h\sigma_x^t(k)f^t(x, y)\sigma_{xy}^t(F^t_{(x,y)} (l,z))f^t(xy, \gyr[x,y](z))t((xy)\gyr[x,y](z)).$

Since $ht(x)((kt(y))(lt(z))) = (ht(x))(kt(y)))\gyr[ht(x),kt(y)](lt(z))$ and $x(yz)= (xy)\gyr[x,y](z)$,  
\begin{equation}\label{1'}
\sigma_x^t(\sigma_y(l)f^t(y, z))f^t(x, yz)= f^t(x, y)\sigma_{xy}^t(F^t_{(x,y)} (l,z))f^t(xy, \gyr[x,y](z)).
\end{equation}
%
%Next, for $x, y \in K$ and $h\in H$,
%$t(x)(t(y)h) = t(x)(\sigma_y^t(h)t(y))= \sigma_x^t(\sigma_y^t(h))f^t(x, y)t(xy)$
%
%On the other hand,
%
%$(t(x)t(y))\gyr[t(x), t(y)](h)) = (f^t(x, y)t(xy))F^t_{(x,y)}((h,1))= f^t(x, y)\sigma_{xy}^t(F^t_{(x,y)}((h,1)))t(xy).$
%
%Since $t(x)(t(y)h) = (t(x)t(y))\gyr[t(x), t(y)](h))$, \begin{equation}\label{1''}
%\sigma_x^t(\sigma_y^t(h))f^t(x, y)= f^t(x, y)\sigma_{xy}^t(F^t_{(x,y)}(h,1)).
%\end{equation}

\begin{remark}\label{trivial} By using the assumption $\gyr[h,g]=I_G$, it is easy to see the following:
\begin{enumerate}
\item If $ht(x)=h't(x)$, then $h=h'$.
\item ${\sigma_x^t}^{-1}(h) = f^t(x^{-1}, x)^{-1}\sigma_{x^{-1}}^t(h)f^t(x^{-1}, x)$, for all $x \in K$ and $h\in H$.
\item $\sigma_x^t(f^t(x^{-1}, x)^{-1})f^t(x, x^{-1}) = e$
\item $\gyr[t(x)h, t(x)^{-1}]=I_G$, for all $x \in K$ and $h\in H$.
\end{enumerate}
\end{remark}

\begin{definition} \label{gyro factor system}
Let $H$ be a group and $K$ be a gyrogroup. Then a \textbf{group-gyro factor system} is a quintuplet $(K,H,\sigma,f,F)$, where $\sigma : K \longrightarrow Aut(H)$ is a map, $f$ is a map from $K \times K$ to $H$ and $F$ is a map from $K\times K$ to $H^{H\times K}$ satisfying the following relations:
\begin{enumerate}
\item $\sigma_e = I_H$;
\item $f(x,e)=e=f(e,y) \ \text{for each}\  x,y\in K$;
\item $f(x,y)\sigma_{xy}(F_{(x,y)}(l, z))f(xy, \gyr[x,y](z))=\sigma_x(\sigma_y(l)f(y,z))f(x,yz) \ \text{for each}\  x,y, z\in K$;
%\item $f(x,y)\sigma_{xy}(F_{(x,y)}(h,1))=\sigma_x(\sigma(h))f(x,y) \ \text{for each}\  x,y\in K$ and $h\in H$;
\item $F_{(x,e)}(l,z)=l=F_{(e,y)}(l,z), \ \forall \ l\in H\ \text{and}\ x,y,z\in H$; 
\item  $F_{(x,y)}(l_1\sigma_{z_1}(l_2)f(z_1,z_2),z_1z_2)\\=F_{(x,y)}(l_1,z_1)\sigma_{\gyr[x,y](z_1))}(F_{(x,y)}(l_2,z_2))f(\gyr[x,y](z_1),\gyr[x,y](z_2))$;
\item $F_{(xy,y)}=F_{(x,y)},\ \forall \ x,y\in K$.

\end{enumerate}
\end{definition}

\begin{proposition}
Every group-gyro extension $E$ of $H$ by $K$ with a choice of a section $t$ determines a group-gyro factor system $(K,H,\sigma^t, f^t, F^t),$ where $\sigma^t$, $f^t$ and $F^t$ are described by Proposition \ref{propF} and equations (\ref{1}), (\ref{3}), (\ref{1'}). 

Conversely, given a group-gyro factor system $(K,H, \sigma, f,F)$, there exists a group-gyro extension $E$ of $H$ by $K,$ and a section $t$ of $E$ such that the group-gyro factor system associated with $E$ is $(K,H,\sigma,f,F).$
\end{proposition}
\begin{proof}
 The proof of the direct part follows from the discussion which motivated the Definition \ref{gyro factor system} and for the converse part, let $G=H\times K.$ Define a product on $G$ by 
 $$(a,x)\cdot (b,y)= (a\sigma_x(b)f(x,y),xy)$$
 and gyro morphism by
 $$\gyr[(a,x),(b,y)](c,z)=(F_{(x,y)}(c,z),\gyr[x,y](z)).$$
 Then $G$ is a gyrogroup and $H$ is a normal subgroup of $G.$ By defining section $t$ by $t(x)=(e,x)$, we have $\sigma^t = \sigma, f^t = f$ and $F^t = F$.
\end{proof}

	\textbf{Equivalence between Category GEXT of group-gyro extension and Category GFAC of group-gyro factor systems-}
 Let  $(\lambda, \mu, \nu)$ be a morphism from the group-gyro extension $E_1$ to the group-gyro extension $E_2$. Then we have the following commutative diagram: 
\[
\xymatrix{
E_1 \equiv \{e\}\ar[r] & H_1 \ar[r]^{i_1} \ar[d]^{\lambda} & G_1 \ar[r]^{\beta_1} \ar[d]^{\mu} & K_1 \ar[r] \ar[d]^{\nu} & \{e\} \\
  E_2 \equiv \{e\}\ar[r] & H_2 \ar[r]^{i_2} & G_2 \ar[r]^{\beta_2} & K_2 \ar[r] & \{e\}
}
\]
Let $t_1$ and $t_2$ be sections of $E_1$ and $E_2$, respectively. Consider the corresponding group-gyro factor systems $(K_1,H_1,f^{t_1},\sigma^{t_1},F^{t_1})$ and $(K_2,H_2,f^{t_2},\sigma^{t_2},F^{t_2})$ of $E_1$ and $E_2$, respectively. Then for $x \in K_1$, $\mu(t_1(x))\in G_2$ and $\beta_2(\mu(t_1(x)))$ = $\nu(\beta_1(t_1(x)))$ = $\nu(x)$ = $\beta_2(t_2(\nu(x)))$. Thus $t_2(\nu(x))\mu(t_1(x))^{-1}$ $ \in  H_2$. In turn, we have a unique $g(x)$ $\in H_2$ such that 
\begin{align*}
g(x)^{-1}= t_2(\nu(x))\mu(t_1(x))^{-1}\\
\Rightarrow g(x)(t_2(\nu(x))\mu(t_1(x))^{-1})=e\ \text{(by left cancellation)}\\
\Rightarrow (g(x)t_2(\nu(x)))\gyr[g(x),t_2(\nu(x))](\mu(t_1(x))^{-1})=e\\
\Rightarrow (g(x)t_2(\nu(x)))\mu(t_1(x))^{-1}=e\\
\Rightarrow \mu(t_1(x))^{-1}(g(x)t_2(\nu(x)))=e.\\
\end{align*}
This implies 
\begin{equation}\label{7}
\mu(t_1(x))=  g(x)t_2(\nu(x)).
\end{equation} 
Since $t_1(e)=e$ = $t_2(e)$,  it follows that
\begin{equation}\label{8}
g(e)=e.
\end{equation}
Since $\mu(t_1(x)t_1(y))=\mu(t_1(x))\mu(t_1(y))$, we have the following equation
\begin{equation}\label{9}
\lambda(f^{t_1}(x,y))g(xy) = g(x)\sigma_{\nu(x)}^{t_2}(g(y))f^{t_2}(\nu(x),\nu(y)).
\end{equation}
Also,
\begin{equation}\label{10}
\mu((t_1(x)h)t_1(x)^{-1})=\lambda(\sigma_x^{t_1}(h)).
\end{equation}
By using the fact that $\mu$ is a gyrogroup homomorphism and Equation \ref{7},
\begin{align*}
\mu((t_1(x)h)t_1(x)^{-1})=\mu(t_1(x))\lambda(h)\mu(t_1(x)^{-1})\\
=(g(x)t_2(\nu(x)))\lambda(h)(t_2(\nu(x))^{-1}g(x)^{-1})\\
=g(x)\sigma_{\nu(x)}^{t_2}(\lambda(h))g(x)^{-1}
\end{align*}
Finally we have
\begin{equation}\label{11}
\mu((t_1(x)h)t_1(x)^{-1})=g(x)\sigma_{\nu(x)}^{t_2}(\lambda(h))g(x)^{-1}.
\end{equation}
On comparing Equation (\ref{10}) and (\ref{11}), we have
\begin{equation}\label{12}
\lambda(\sigma_x ^{t_1}(h))=g(x)\sigma_{\nu(x)}^{t_2}(\lambda(h))g(x)^{-1}.
\end{equation}
We know that,
$$\gyr[ht_1(x),kt_1(y)](lt_1(z))=F^{t_1}_{(x,y)}(l,z)t_1(\gyr[x,y](z))$$
\begin{align*}
\mu(\gyr[ht_1(x),kt_1(y)](lt_1(z)))=\gyr[\lambda(h)\mu(t_1(x)),\lambda(k)\mu(t_1(y))](\lambda(l)\mu(t_1(z)))\\
=\gyr[\lambda(h)g(x)t_2(\nu(x)),\lambda(k)g(y)t_2(\nu(y))](\lambda(l)g(z)t_2(\nu(z)))\\
=\gyr[(\lambda(h)g(x))t_2(\nu(x)),(\lambda(k)g(y))t_2(\nu(y))]((\lambda(l)g(z))t_2(\nu(z)))\\
=F^{t_2}_{(\nu(x),\nu(y))}(\lambda(l)g(z)),\nu(z))t_2(\gyr[\nu(x),\nu(y)](\nu(z))).
\end{align*}
We have the following equation
\begin{equation}\label{13}
\mu(\gyr[ht_1(x),kt_1(y)](lt_1(z)))=F^{t_2}_{(\nu(x),\nu(y))}(\lambda(l)g(z)),\nu(z))t_2(\gyr[\nu(x),\nu(y)](\nu(z))).
\end{equation} 
Also we have,
\begin{align*}
\mu(\gyr[ht_1(x),kt_1(y)](lt_1(z)))=\lambda(F^{t_1}_{(x,y)}(l,z))\mu(t_1(\gyr[x,y](z)))\\
=\lambda(F^{t_1}_{(x,y)}(l,z))g(gyr[x,y](z))t_2(\nu(\gyr[x,y](z)).
\end{align*}
Finally, we have
\begin{equation}\label{14}
\mu(\gyr[ht_1(x),kt_1(y)](lt_1(z)))=\lambda(F^{t_1}_{(x,y)}(l,z))g(gyr[x,y]z)t_2(\gyr[\nu(x),\nu(y)](\nu(z))).
\end{equation}
On comparing Equation (\ref{13}) and (\ref{14}),
\begin{equation}\label{15}
F^{t_2}_{(\nu(x),\nu(y))}(\lambda(l)g(z),\nu(z))=\lambda(F^{t_1}_{(x,y)}(l,z))g(gyr[x,y](z)).
\end{equation}
Thus a morphism $(\lambda,\mu,\nu)$ between two group-gyro extensions $E_1$ and $E_2$ together with choices of sections $t_1$ and $t_2$ of the corresponding extensions, induces a map $g$ from $K_1$ to $H_2$  such that the triple $(\nu, g, \lambda)$ satisfies equations (\ref{8}), (\ref{9}), (\ref{12}) and (\ref{15}). It can be seen as a morphism from the factor system  $(K_1,H_1,f^{t_1},\sigma^{t_1},F^{t_1})$ to $(K_2,H_2,f^{t_2},\sigma^{t_2},F^{t_2})$.

Let $(\lambda_1,\mu_1,\nu_1)$ be a morphism from the group-gyro extension
\[
\xymatrix{
E_1 \equiv \{e\}\ar[r] & H_1\ar[r]^{i_1} & G_1 \ar[r]^{\beta_1} & K_1\ar[r] & \{e\}
}
\]
to the group-gyro extension
\[
\xymatrix{
E_2 \equiv \{e\}\ar[r] & H_2\ar[r]^{i_2} & G_2 \ar[r]^{\beta_2} & K_2\ar[r] & \{e\}
}
\]
and $(\lambda_2,\mu_2,\nu_2)$ be another morphism from the group-gyro extension $E_2$ to the group-gyro extension $E_3$
\[
\xymatrix{
E_3 \equiv \{e\}\ar[r] & H_3\ar[r]^{i_3} & G_3 \ar[r]^{\beta_3} & K_3\ar[r] & \{e\}
}.
\]
Let $t_1, ~ t_2,~\text{and} ~ t_3$ be corresponding choices of sections. Then
$\mu_1(t_1(x))=g_1(x)t_2(\nu_1(x))$ and $\mu_2(t_2(x))=g_2(x)t_3(\nu_2(x))$, where $g_1:K_1 \longrightarrow H_2$ and $g_2:K_2 \longrightarrow H_3$ are uniquely determined maps same as $g$ in Equation $(\ref{7})$.
In turn, we have 
\begin{align*}
\mu_2(\mu_1(t_1(x)))&=\lambda_2(g_1(x))\mu_2(t_2(\nu_1(x)))\\
&=\lambda_2(g_1(x)))g_2(\nu_1(x))t_3(\nu_2(\nu_1(x)))\\
 &= \lambda_2 (g_1(x))g_2(\nu_1(x))t_3(\nu_2(\nu_1(x)))\\
 &=g_3(x)t_3(\nu_2(\nu_1(x))),\\
  \text{where}\ g_3(x) = \lambda_2 (g_1(x))g_2(\nu_1(x)),\ \text{for each}\ x \in K_1.
\end{align*}
Thus the composition $(\lambda_2 \circ \lambda_1, \mu_2 \circ \mu_1,\nu_2 \circ \nu_1)$ induces the triple $(\nu_2 \circ \nu_1,g_3, \lambda_2 \circ \lambda_1)$.

Now we introduce the category GFAC  whose objects are group-gyro factor system, and a morphism from $(K_1,H_1,f^{1},\sigma^1,F^1)$ to $(K_2,H_2,f^{2},\sigma^2,F^2)$ is a triple $(\nu, g, \lambda)$, where $\nu : K_1 \longrightarrow K_2$ is gyrogroup  homomorphism, $\lambda: H_1 \longrightarrow H_2$ is group  homomorphism, and $g: K_1 \longrightarrow H_2$ is a map such that 
\begin{enumerate}
\item $g(e) = e$;
\item $\lambda(f^{1}(x,y))g(xy)$ = $g(x)\sigma_{\nu(x)}^{2}(g(y))f^{2}(\nu(x),\nu(y))$;
\item $\lambda(\sigma_x ^{1}(h))=g(x)\sigma_{\nu(x)}^{2}(\lambda(h))g(x)^{-1}$;
\item $F^{2}_{(\nu(x),\nu(y))}(\lambda(l)g(z),\nu(z))=\lambda(F^{1}_{(x,y)}(l,z))g(gyr[x,y](z))$.

\end{enumerate}
The composition of morphisms $(\nu_1, g_1, \lambda_1)$ from $(K_1,H_1,f^{1},\sigma^1,F^1)$ to $(K_2,H_2,f^{2},\sigma^2,F^2)$ and $(\nu_2, g_2, \lambda_2)$ from $(K_2,H_2,f^{2},\sigma^2,F^2)$ to $(K_3,H_3,f^{3},\sigma^3,F^3)$ is the triple $(\nu_2 ~ o~ \nu_1, g_3, \lambda_2 ~ o~ \lambda_1)$, where $g_3$ is given by $g_3 (x) =  \lambda_2 (g_1(x))g_2(\nu_1(x))$
for each $x \in K_1$.

So, finally from the above discussion, we have the following Theorem:
\begin{theorem} \label{equivalence}

There is an equivalence between the category $\mathbf{GEXT}$ of group-gyro extensions to the category $\mathbf{GFAC}$ of group-gyro factor systems.
\end{theorem}

	\textbf{Dependency of an extension on sections-} Let $s$ and $t$ be two sections of an extension \[\xymatrix{
E \equiv \{e\}\ar[r] & H\ar[r]^{i} & G \ar[r]^{\beta} & K\ar[r] & \{e\}
}.
\] Then there exists an identity preserving map $ g : K \rightarrow H$ such that $s(x) = g(x)t(x)$ (see Equation (\ref{7})). Hence by taking $\lambda = I_H$, $\mu = I_G$ and $\nu = I_K$ in equations (\ref{9}), (\ref{12}), (\ref{15}), we have the following identities:
\begin{enumerate}
\item $f^{s}(x,y)g(xy)$ = $g(x)\sigma_{x}^{t}(g(y))f^{t}(x,y)$.
\item $\sigma_x ^{s}(h))=g(x)\sigma_{x}^{t}(h)g(x)^{-1}$.
\item $F^{t}_{(x,y)}(lg(z),z)=F^{s}_{(x,y)}(l,z))g(gyr[x,y](z))$.
\end{enumerate}

\section{Semi cross product and examples of gyrogroups}
In this section, we define semi cross product of a group and a gyrogroup. By using  semi cross product, we construct few examples of gyrogroups of finite and infinite orders.
\begin{proposition}\label{semi}
Let $H$ be a group and $K$ be a gyrogroup. Let $\sigma: K \longrightarrow Aut(H)$ be a map such that
\begin{enumerate}
\item $\sigma_1 = I_H$, where $I_H$ denotes the identity map on $H$;
\item $\sigma_{x^{-1}} = \sigma^{-1}_x$;
\item $\sigma_{((xy)y)^{-1}}\circ \sigma_{xy} = \sigma_{(xy)^{-1}}\circ \sigma_x$.
\end{enumerate}
for all $x, y \in K$, where $\sigma_x$ denotes the image of $x$ under $\sigma$. Then $H \times K$ is a gyrogroup  with the binary operation $*$
$$(h, x)*(k, y) = (h\sigma_x (k), xy)$$ and the gyroautomorphism $$\gyr[(h, x), (k, y)] (l, z) = ((\sigma_{(xy)^{-1}}\circ \sigma_x\circ \sigma_y)(l), \gyr[x, y](z)).$$ We call it as a semi cross product of $H$ and $K$, and we denote it by $H\Join K$.
\end{proposition}
\begin{proof}
\begin{enumerate}
\item \textbf{Identity:} $(e, e)$ is the identity element. 
\item \textbf{Inverse:} Let $(h, x) \in H \times K$. Then $(\sigma_{x^{-1}} (h^{-1}), x^{-1})$ is the inverse.
\item \textbf{Left gyroassociative law:} Let $(h, x), (k, y), (l, z) \in H \times K$. Then 

$(h, x)* ((k, y)*(l, z)) = (h, x)*(k\sigma_y (l), yz) = (h \sigma_x(k\sigma_y (l)), x(yz))$ 

On the other hand, 

$((h, x)* (k, y))*\gyr[(h, x), (k, y)] (l, z) = (h \sigma_x(k), xy)*((\sigma_{(xy)^{-1}}\circ \sigma_x\circ \sigma_y)(l), \\\gyr[x, y](z)) = (h \sigma_x(k) (\sigma_{(xy)}\circ\sigma_{(xy)^{-1}}\circ \sigma_x\circ \sigma_y)(l), (xy)\gyr[x, y](z))= (h \sigma_x(k\sigma_y (l)), x(yz))$.

Therefore, $(h, x)* ((k, y)*(l, z))= ((h, x)* (k, y))*\gyr[(h, x), (k, y)] (l, z)$.

\item \textbf{Left loop property:} $\gyr[(h, x)*(k, y), (k, y)] = \gyr[(h \sigma_x(k), xy), (k, y)] = (\sigma_{((xy)y)^{-1}}\circ \sigma_{xy}\circ \sigma_{y}, \gyr[xy, y]) = (\sigma_{(xy)^{-1}}\circ \sigma_x\circ \sigma_{y}, \gyr[x, y])$ (by the loop identity in $K$).
\end{enumerate}
\end{proof}

\begin{definition}
An extension \[
\xymatrix{
E \equiv \{e\}\ar[r] & H\ar[r]^{i} & G \ar[r]^{\beta} & K\ar[r] & \{e\}
}
\] of $H$ by $K$ is called an split extension if there is a section $t$ which is a gyrogroup homomorphism. Such a section $t$ is called a splitting of the extension. The corresponding factor system $(K,H,f^{t},\sigma^t,F^t)$ is such that $f^{t}$ is trivial in the sense that $f^t(x, y) = e$, for all $x, y \in K$.
\end{definition}
\begin{remark}\label{imp}
\begin{enumerate}
\item If $K$ is a group and $\sigma$ is a group homomorphism, then the semi cross product is the semi direct product of $H$ and $K$.
\item Suppose $Aut(H)$ is an abelian group such that every element is its own inverse. Let $\sigma: K \longrightarrow Aut(H)$ be a map which satisfies the first two conditions of Proposition \ref{semi}. Then the third condition becomes $\sigma_{((xy)y)} = \sigma_x$, for all $x, y \in K$. Moreover, if $K$ is a group, then  the third condition becomes $\sigma_{xy^2} = \sigma_x$, for all $x, y \in K$.
\item Suppose $K = \{0, 1, 2, 3, 4, 5, 6, 7\}$ is the gyrogroup given in Example 3.2 in \cite{BKS}. The table for the element $(xy)y$ is as follows:\\

\begin{tabular}{ c  |c  c  c c c c c c}
        \hline
             &0 & 1 & 2 & 3& 4& 5& 6& 7\\\hline
            0 & 0 & 0 & 0 & 0& 0& 0& 0& 0\\ 
            1 & 1 & 1 & 6 & 1& 1& 6& 1& 1\\
            2 & 2 & 5 & 1 & 2& 2& 2& 5& 2\\
            3 & 3 & 4 & 4 & 3& 3& 4& 4& 3\\
            4 & 4 & 3 & 3 & 4& 4& 3& 3& 4\\
            5 & 5 & 2 & 5 & 5& 5& 5&2& 5\\
            6 & 6 & 6 & 1 & 6& 6& 1& 6& 6\\
            7 & 7 & 7 & 7 & 3& 7& 7& 7& 7\\           
        \hline
        \end{tabular}\\
        
The first column and the first row stand for the values of $x$ and $y$, respectively.

\item Suppose $K$ is  the quaternion group $Q_8 = \{1, -1, i, -i, j, -j, k, -k\}$. The table for the element $(xy)y= xy^2$ is as follows:\\

\begin{tabular}{ c  |c  c  c c c c c c}
        \hline
             &1 & -1 & i & -i& j& -j& k& -k\\\hline
           1 & 1 & 1 & -1 & -1& -1 & -1& -1 & -1\\ 
            -1 & -1 & -1 & 1 & 1& 1 & 1& 1 & 1\\
            i & i & i & -i & -i& -i & -i& -i & -i\\
            -i & -i & -i & i & i& i & i& i & i\\
            j & j & j & -j & -j& -j & -j& -j & -j\\
            -j & -j & -j & j & j& j & j& j & j\\
            k & k & k & -k & -k& -k & -k& -k & -k\\
            -k & -k & -k & k & k& k & k& k & k\\           
        \hline
        \end{tabular}\\

The first column and the first row stand for the values of $x$ and $y$, respectively.
\end{enumerate}

\end{remark}

By Proposition \ref{semi}, one can see the following results analogous to group extension.
\begin{theorem} Let $H$ be a group and $K$ be a gyrogroup. Then
$G= H\Join K$ if and only if the following conditions  hold:
\begin{enumerate}
\item $H \cong  \{(h, e)|h \in H\}$ and $K \cong  \{(e, k)|k \in K\}$;
\item $G = HK$;
\item $\gyr[(h, e), (h', k')]  = I_G$;
\item $((h_1, x)(h_2, e))(h_1, x)^{-1}\in H$.
\end{enumerate}
\end{theorem}

\begin{theorem}
Let $H$ be a group, $K$ be a gyrogroup and \[
\xymatrix{
E \equiv 1\ar[r] & H\ar[r]^{i} & G \ar[r]^{\beta} & K\ar[r] & 1
}
\] of $H$ by $K$ be an split extension. Then $G  \cong H\Join K$. Conversely, if $G  = H\Join K$, then there is a natural projection $p$ from $G$ to $K$ such that \[
\xymatrix{
E \equiv \{e\}\ar[r] & H\ar[r]^{i} & G \ar[r]^{\beta} & K\ar[r] & \{e\}
}
\] is a split extension of $H$ by $K$.
\end{theorem}
\begin{proof}
Since $E$ is an split extension, we have a section $t$ which is gyrogroup homomorphism. This implies that $f^t(x, y) = e$, for all $x, y\in K$. Therefore, by Remark \ref{trivial} and Proposition \ref{propF}, we have $\sigma_{x^{-1}} = \sigma^{-1}_x$ and $\sigma_{((xy)y)^{-1}}\circ \sigma_{xy} = \sigma_{(xy)^{-1}}\circ \sigma_x$, for all $x, y \in K$. The remaining part of the proof is easy to see.
\end{proof}
\begin{example}\label{example1}\textbf{[Gyrogroups of order 24]}
\begin{enumerate}
\item Let $H= \mathbb{Z}_3 = \{\bar{0}, \bar{1}, \bar{2}\}$ be the cyclic group of order  $3$ and $K = \{0, 1, 2, 3, 4, 5, 6, 7\}$ be the gyrogroup given in Example 3.2 in \cite{BKS}. Then $\text{Aut}(H) = \{I_H, f\}$, where $f: H \longrightarrow H$ defined by 
$f(\bar{0}) = \bar{0}, f(\bar{1}) = \bar{2}, f(\bar{2}) = \bar{1}$.
\begin{enumerate}
\item If $\sigma : K \longrightarrow \text{Aut}(H)$ is the trivial map ($\sigma(x) = I_H$, for all $x$), then $G= H \Join K$ is a gyrogroup of order $24$ with the given binary operation and gyroautomorphisms:
$$(h, x)*(k, y) = (hk, xy)$$ and $$\gyr[(h, x), (k, y)] = (I_H, \gyr[x, y]).$$
\item Let $\sigma : K \longrightarrow \text{Aut}(H)$ be a map defined by 
\[ \sigma(x) = \begin{cases}
f,  ~\text{if}  ~ x = 7\\
I_H,  ~\text{otherwise}.
\end{cases}\]

Clearly, $\sigma(0) = I_H$ and $\sigma(x^{-1}) = \sigma^{-1}(x)$, for all $x$. By the definition of $K$, every element of $K$ is its own inverse. Since $f$ is its own inverse, $\sigma(x^{-1}) = \sigma(x)$. Hence, by Remark \ref{imp} (2) and (3), $\sigma_{((xy)y)} = \sigma_x ~ \forall~  x, y \in K$. Thus $G= H\Join K$ is a gyrogroup of order $24$. More precisely, the binary operation on $G$ and gyroautomorphisms are given below:
\[(h, x)*(k, y) = \begin{cases}
(hf(k), 7y),  ~\text{if}  ~ x = 7\\
(hk, xy),  ~\text{otherwise}
\end{cases} \]
 and 
 \[\gyr[(h, x), (k, y)] = \begin{cases}
(f, I),  ~\text{if}  ~ (x, y)\in X \\
(I_H, A),  ~(x, y)\in Y
\\
(I_H, I),  ~\text{otherwise}
\end{cases} \]
\begin{align*}
\text{where}~ &X = \{(1, 6), (6, 1), (2, 5), (5, 2), (3, 4), (4, 3)\} ~\text{and}\\
&Y = \{(1, 2), (1, 3), (1, 4), (1, 5), (2, 1), (2, 3), (2, 4), (2, 6), (3, 1), (3, 2), (3, 5), (3, 6), (4, 1), (4, 2),\\& (4, 5), (4, 6), (5, 1), (5, 3), (5, 4), (5, 6), (6, 2), (6, 3), (6, 4), (6, 5)\}.
\end{align*}

\end{enumerate}
\item Let $H= \mathbb{Z}_3 = \{\bar{0}, \bar{1}, \bar{2}\}$ be the cyclic group of order  $3$ and $K = Q_8$, where $Q_8$ is the quaternion group of order $8$. Then $\text{Aut}(H) = \{I_H, f\}$, where $f: H \longrightarrow H$ defined by 
$f(\bar{0}) = \bar{0}, f(\bar{1}) = \bar{2}, f(\bar{2}) = \bar{1}$. Define a map $\sigma : K \longrightarrow \text{Aut}(H)$ by 
\[ \sigma(x) = \begin{cases}
f,  ~\text{if}  ~ x = \pm i\\
I_H,  ~\text{otherwise}.
\end{cases}\]

Clearly, $\sigma(0) = I_H$ and $\sigma(x^{-1}) = \sigma^{-1}(x)$, for all $x$. Since $f$ is its own inverse, $\sigma(x^{-1}) = \sigma(x)$. Hence by Remark \ref{imp} (2) and (3), $\sigma_{xy^2} = \sigma_x ~ \forall~  x, y \in K$. Thus $G= H\Join K$ is a gyrogroup of order $24$. More precisely, the binary operation on $G$ and gyroautomorphisms are given below:
\[(h, x)*(k, y) = \begin{cases}
(hf(k), xy),  ~\text{if}  ~ x = \pm i\\
(hk, xy),  ~\text{otherwise}
\end{cases} \]
 and 
 \[\gyr[(h, x), (k, y)] = (\sigma_{xy} \circ \sigma_x\circ \sigma_y, I_K)= \begin{cases}
(f, I_K),  ~\text{if}  ~ (x, y)\in A \\
(I_H, I_K),  ~\text{otherwise}
\end{cases} \]
where $A$ is \begin{align*}
&\{(j, k), (j, -k), (-j, k), (-j, -k), (k, j), (k, -j), (-k, j), (-k, -j), (i, k), (i, -k), (-i, k), (-i, -k), \\&(i, j), (i, -j), (-i, j), (-i, -j), (j, k), (j, -k), (-j, k), (-j, -k), (k, j), (k, -j), (-k, j), (-k, -j)\}.
\end{align*}
\end{enumerate} 
\end{example}

\begin{example}\textbf{[Gyrogroups of order 32]}
Let $H= \mathbb{Z}_4 = \{\bar{0}, \bar{1}, \bar{2}, \bar{3}\}$ be the cyclic group of order  $4$ and $K = \{0, 1, 2, 3, 4, 5, 6, 7\}$ be the gyrogroup given in Example 3.2 in \cite{BKS}. Then $\text{Aut}(H) = \{I_H, f\}$, where $f: H \longrightarrow H$ defined by 
$f(\bar{0}) = \bar{0}, f(\bar{1}) = \bar{3}, f(\bar{2}) = \bar{2}, f(\bar{3}) = \bar{1}$. 

The maps $\sigma$ from $K$ to $\text{Aut}(H)$ given in Example \ref{example1} 1(a) and 1(b)  will give two gyrogroups of order 32. Also, the map $\sigma$ from $Q_8$ to $\text{Aut}(H)$ given in Example \ref{example1} (2)  will give another gyrogroup of order 32.
\end{example}
\begin{example}\textbf{[Gyrogroup of infinite order]}
Let $H= \mathbb{Z}$ be the group of integers and $K = \{0, 1, 2, 3, 4, 5, 6, 7\}$ be the gyrogroup given in Example 3.2 in \cite{BKS}. Then $\text{Aut}(H) = \{I_H, f\}$, where $f: H \longrightarrow H$ defined by 
$f(x) = -x ~\forall~ x \in H$. The maps $\sigma$ from $K$ to $\text{Aut}(H)$ given in Example \ref{example1} 1(a) and 1(b)  will give two gyrogroups of infinite order. Also, the map $\sigma$ from $Q_8$ to $\text{Aut}(H)$ given in Example \ref{example1} (2)  will give another gyrogroup of infinite order.
\end{example}

\textbf{Conclusion:} We have discussed the theory of extensions of a group by a gyrogroup which is the first question posed in the introduction. In continuation of this, we define semi cross product of a group and a gyrogroup, and constructed several examples of gyrogroups. Note that one can construct a gyrogroup  (which is not a group) with the help of two groups which is evident in the examples. The semi cross product may be helpful to classify finite gyrogroups upto isomorphism.  

{\bf Acknowledgment:} We are extremely thankful to Prof. Ramji Lal for his continuous support, discussion and encouragement. The first named author thanks IIIT Allahabad and Ministry of Education, Government of India for providing institute fellowship. The third named author is thankful to IIIT Allahabad for providing SEED grant.

%	\bibliographystyle{plain}
	%	\bibliography{bib} 		

\end{document}